\documentclass[12pt]{amsart}
\usepackage{latexsym}
\usepackage{amssymb}
\usepackage{amsmath}
\usepackage[mathscr]{eucal}

\def\rest{\restriction}
\def\a{\alpha}
\def\b{\beta}
\def\g{\gamma}
\def\o{\omega}
\def\d{\delta}

\def\t{\theta}
\def\morass{\mathfrak{M}}

\newcommand{\forces}{\Vdash}

\newcommand{\ef}{Ehrenfeucht-Fra\"\i{}ss\'e}

\title[Long Ehrenfeucht-Fra\"\i{}ss\'e games]{Positional strategies in long Ehrenfeucht-Fra\"\i{}ss\'e games}
\author{S. Shelah}
\address{Institute of Mathematics, Hebrew University,
Jerusalem, Israel
and
Department of Mathematics, Rutgers University, New Brunswick, NJ, USA}
\email{shelah@math.huji.ac.il}
\urladdr{http://shelah.logic.at}
\thanks{The first author would like to thank the Israel Science Foundation for partial support of this research. Publication XXX}

\author{J. V\"a\"an\"anen}
\address{Department of Mathematics and Statistics, University of Helsinki, Finland
and Institute for Logic, Language and Computation University of Amsterdam, The Netherlands}
\email{jouko.vaananen@helsinki.fi}
\urladdr{http://www.math.helsinki.fi/logic/people/jouko.vaananen}
\thanks{Research of the second author was partially supported by grant 40734 of the Academy of Finland and
by the Eurocores LogICCC LINT programme}

\author{B. Veli\v ckovi\'c}
\address{Equipe de Logique Math\'ematique,
Institut de Math\'ematiques de Jussieu,
Universit\'e Paris Diderot, Paris, France}
\email{boban@math.univ-paris-diderot.fr}
\urladdr{http://www.logique.jussieu.fr/~ boban}

\begin{document}

\bibliographystyle{plain}

\newtheorem{theorem}{Theorem}[section]
\newtheorem{lemma}[theorem]{Lemma}
\newtheorem{proposition}[theorem]{Proposition}
\newtheorem{definition}[theorem]{Definition}
\newtheorem{remarks}[theorem]{Remarks}
\newtheorem{remark}[theorem]{Remark}
\newtheorem{examples}[theorem]{Examples}
\newtheorem*{clai}{Claim}

\newtheorem{question}{Question}
\newcommand{\Q}{\mathbb{Q}}
\newcommand{\A}{\mathcal{A}}
\newcommand{\M}{\mathcal{M}}
\newcommand{\B}{\mathcal{B}}
\newcommand{\F}{\mathcal{F}}
\newcommand{\I}{\mathcal{I}}
\newcommand{\G}{\mathcal{G}}
\newcommand{\LL}{\mathcal{C}}
\newcommand{\C}{{\mathcal C}}
\newcommand{\N}{\mathcal{N}} 
\newcommand{\Pp}{\mathbb{P}} 
\newcommand{\Pcal}{\mathcal{P}}
\newcommand{\EF}{\mathrm{EF}}
\newcommand{\pla}{\forall} 
\newcommand{\ple}{\exists} 
\def\force{\Vdash}
\newcommand{\dom}[1]{\mathrm{dom}(#1)}
\newcommand{\ran}[1]{\mathrm{ran}(#1)}
\newcommand{\pisoxy}[2]{\simeq_{#1, #2}^{p}} 
\newcommand{\pisox}[1]{\simeq_{#1}^{p}}
\newcommand{\pisoxx}[1]{\pisoxy{#1}{#1}}
\newcommand{\pisok}{\pisox{\kappa}}
\newcommand{\pisokl}{\pisoxy{\kappa}{\lambda}}
\newcommand{\pisoaa}{\pisoxx{\aleph_1}}
\newcommand{\plexy}[2]{(\ple)_{#1}^{#2}} 
\newcommand{\plekl}{\plexy{\kappa}{\lambda}}
\newcommand{\pleaa}{\plexy{\aleph_1}{\aleph_1}}
\newcommand{\ee}[1]{\equiv_{#1}} 
\newcommand{\answer}{\emph{Answer.}\quad}
\newcommand{\proofsketch}{\noindent\emph{Sketch of proof.}\quad}
\newcommand{\st}{such that}
\def\qed{$\Box$\medskip

}

\def\ma{{\mathcal A}}
\def\mb{{\mathcal B}}
\def\mm{{\mathcal S}}
\newcommand\strong[2]{\simeq^{p}_{#1,#2}}
\newcommand\weak[1]{\simeq^{p}_{#1}}

\begin{abstract}
We prove that it is relatively consistent with ${\rm ZF + CH}$ that there exist two models
of cardinality $\aleph_2$ such that the second player has a winning strategy in the  \ef-game of
length $\omega_1$  but there is no $\sigma$-closed back-and-forth set for the two models. If ${\rm CH}$
fails, no such pairs of models exist.
\end{abstract}

\maketitle

\section{Introduction}

 Suppose $\ma = (A,\ldots)$ and $\mb=(B,\ldots)$ are structures for the same  vocabulary $\mathcal L$ of
 cardinality $<\kappa$. We say that a set $\mathcal I$ of partial isomorphisms between $\mathcal A$ and $\mathcal B$
has the $\kappa$-{\em back-and-forth property}  if for every $p\in \mathcal I$, and every
$A_0\subseteq A$ and  $B_0\subseteq B$ of size $<\kappa$ there is $q\in \mathcal I$ extending
$p$ such that $A_0\subseteq \dom q$ and $B_0\subseteq \ran q$.
We say that $\ma$ and $\mb$ $\kappa$-{\em partially isomorphic} and
write $\ma\weak{\kappa}\mb$ if there is a $\kappa$-back-and-forth set for \(\ma\) and \(\mb\).
The relation \(\ma\weak{\kappa}\mb\) has a metamathematical interpretation.
Namely, for regular $\kappa$ it coincides with elementary equivalence relative to
the infinitary language $L_{\infty\kappa}$. In particular, $\weak{\kappa}$ is an
equivalence relation on  the class of all $\mathcal L$-structures. If $\kappa$
is uncountable then even for models of cardinality $\kappa$ the relation
$\weak{\kappa}$ is strictly weaker than isomorphism. This was first proved by Morley (1968, unpublished,
see~\cite{MR57:2907}). For instance, for $\kappa =\aleph_1$, one can take
a pair of $\aleph_1$-like dense linear orders
one of which contains a closed copy of $\omega_1$ while the other doesn't.


In this paper we investigate a strengthening of the relation $\weak{\kappa}$.
Namely, given two cardinals $\kappa$ and $\lambda$ and two structures $\mathcal A$ and $\mathcal B$
in a vocabulary of size $<\kappa$, we say that $\mathcal A$ and $\mathcal B$ are
$(\kappa,\lambda)$-{\em partially isomorphic} and write
\(\ma\strong{\kappa}{\lambda}\mb\) if there is a
\(\kappa\)-back-and-forth set $\mathcal I$ between \(\ma\) and \(\mb\) such
that any increasing chain of length \(<\lambda\) in $\mathcal I$ has an
upper bound in $\mathcal I$. The point is that the relation \(\strong{\kappa}{\kappa}\),
unlike the weaker version \(\weak{\kappa}\), implies isomorphism in the case that the models are
of cardinality at most $\kappa$, and many classical isomorphism-proofs can be
interpreted as results about the relation \(\strong{\kappa}{\lambda}\).
Indeed, suppose \(\kappa\) is regular. Then any two
\(\eta_{\kappa}\)-sets are in the relation
\(\strong{\kappa}{\kappa}\). If they are of cardinality
\(\kappa\), they are isomorphic. Also, it is well known that any two real closed fields
whose underlying orders are of type \(\eta_{\omega_1}\) and are of cardinality $\omega_1$
are isomorphic, see \cite{MR16:993e}. In fact, if \(\kappa\) is regular then any two
real closed fields whose underlying orders are of type \(\eta_{\kappa}\) are in the relation
\(\strong{\kappa}{\kappa}\), see \cite{MR58:27456}. Another example concerns saturated
models. Any two \(\kappa\)-saturated elementary equivalent structures of
cardinality \(\kappa\) are isomorphic, and the proof shows that
any two \(\kappa\)-saturated elementary equivalent structures are
in the relation \(\strong{\kappa}{\kappa}\). Finally, consider
two \(\kappa\)-homogeneous structures $\mathcal A$ and $\mathcal B$ such that \(\ma\weak{\kappa}\mb\).
If they happen to be of cardinality \(\kappa\)
they are isomorphic and the proof goes by showing that \(\ma\strong{\kappa}{\kappa}\mb\).

\medskip

Thus, the relation $\strong{\kappa}{\kappa}$ seems like an
attractive weaker version of isomorphism. However, there are some
 simple questions concerning it that are still open. The most important one
 was raised by Dickmann~\cite{MR58:27450} and
Kueker~\cite{MR57:2905}, and asks if $\strong{\kappa}{\kappa}$ is equivalent
to elementary equivalence in some logic. In fact, it is not  clear if
$\strong{\kappa}{\kappa}$ is even transitive.
This was a serious obstacle to generalizing first order logic. In order to overcome
this Karttunen~\cite{MR82k:03051} defined {\em tree-like  partial isomorphisms}.
This leads to a transitive relation which coincides with elementary equivalence in a
certain logic called $\N_{\infty\kappa}$ and implies isomorphism for models of size $\kappa$.
One can translate Karttunen's concept in terms of  the existence of a winning strategy
in a certain Ehrenfeucht-Fra\"\i{}ss\'e game which we now describe.
To begin, we fix two regular cardinals $\kappa$ and $\lambda$
and two structures $\mathcal A$ and $\mathcal B$ in the same vocabulary $\mathcal L$
of size $<\kappa$.

\begin{definition}[$\EF_\kappa^\lambda(\A,\B)$]
  There are two players $\pla$ and $\ple$.
The game runs in $\lambda$ rounds and proceeds as follows.

{\setlength\arraycolsep{0.2em}
$$\begin{array}{c|ccccccc}
  \forall \; & A_0,B_0 & && \ldots &A_\alpha, B_\alpha&&\ldots\\
  \hline
  \exists \; & &p_0 & & \ldots &&p_\alpha&\ldots\\
\end{array}\hspace{1cm}(\alpha<\lambda)$$
}
\medskip
  \noindent At stage $\alpha < \lambda$, player $\pla$ picks $A_\alpha \subseteq A$
  and $B_\alpha \subseteq B$, both of size  $<\kappa$.
  Player $\ple$ responds by  a partial isomorphism $p_\alpha$ between a substructure
  of $\mathcal A$ of size $<\kappa$ containing $A_\alpha$ and a substructure of $\mathcal B$
  containing $B_\alpha$. We require that $p_\alpha$ extends the $p_\xi$, for $\xi <\alpha$.
  Player $\ple$ wins the game if she plays $\lambda$ rounds
  while obeying the rules. Otherwise player $\pla$ wins.
\end{definition}

We write $\A \equiv_{\kappa,\lambda}\B$  if $\ple$ has a winning
strategy in $\EF_\kappa^\lambda(\A,\B)$. This is clearly transitive.
This concept has allowed the study of infinitary languages to take off and has been
very fruitful (see e.g. \cite{MR2768176}). One of the first new results
was obtained by Hyttinen~\cite{MR88a:03086} who proved the Craig
Interpolation Theorem and other classical results for this new
logic. Still the following question remains.

\begin{question}\label{main_question}
  What is the relation between $\pisokl$ and $\equiv_{\kappa,\lambda}$?
\end{question}

Clearly, if $\ma\pisokl\mb$ then  $\A \equiv_{\kappa,\lambda}\B$. Indeed, if $\ma\strong{\kappa}{\lambda}\mb$
then there is a {\em positional} winning strategy for $\exists$ in $\EF_\kappa^\lambda(\A,\B)$,
in the sense that $\exists$ only needs to know the current position in order to know how to play and win.
Thus, Question \ref{main_question} simply asks if the converse is true.
Note that the positive answer implies that $\strong{\kappa}{\lambda}$ is transitive.
We concentrate on the first nontrivial case, namely the relation between
$\simeq^p_{\aleph_1,\aleph_1}$ and $\equiv_{\aleph_1,\aleph_1}$.
Let us first note the well known fact that $\A \equiv_{\aleph_1,\aleph_1}\B$ can be expressed
as the existence of  {\em potential isomorphism\footnote{Recall that for purely relational structures
 $\A\equiv_{\omega,\omega}\B$ is equivalent to the existence of an isomorphism of $\ma$ and $\mb$
 in {\it some} forcing extension.}}
an isomorphism in a forcing extension obtained by $\sigma$-closed forcing.

\begin{proposition}
Suppose $\A$ and $\B$ are structures in the same vocabulary $\mathcal L$.
Then $\A \equiv_{\aleph_1,\aleph_1}\B$ if and only if there is a $\sigma$-closed
forcing notion $\mathcal P$ such that $\forces_{\mathcal P}\A \cong \B$.
\qed
\end{proposition}

We recall the following results from \cite{vv} where the equivalence of
$\simeq^p_{\aleph_1,\aleph_1}$ and $\equiv_{\aleph_1,\aleph_1}$
has been established in some special cases.

\begin{theorem}
  Suppose $\A$ and $\B$ are two structures in the same vocabulary $\mathcal L$.
  Then $\A\simeq^p_{\aleph_1,\aleph_1}\B$ and
  $\A \equiv_{\aleph_1,\aleph_1} \B$ are equivalent in any of the following cases:
  \begin{enumerate}

\item  $|\A|,|\B|\leq 2^{\aleph_0}.$
\item  $\A$ and $\B$ have different cardinality.
\item $\A$ and $\B$ are trees of height $\aleph_1$. \qed
\end{enumerate}

\end{theorem}

On the basis of these results it seems interesting to investigate the case when
$\A$ and $\B$ are of cardinality $\aleph_2$ and $\mathrm{CH}$ holds. Even in this case we can  have
a positive result if we look at partial isomorphisms of size
$\aleph_1$ rather than of size $\aleph_0$. The following result was proved in  \cite{vv}.

\begin{theorem}
  Suppose $\A$ and $\B$ are two structures of cardinality $\aleph_2$ in the same vocabulary $\mathcal L$.
  Then $\A \simeq^p_{\aleph_2,\aleph_1}\B$ if and only if $\A \equiv_{\aleph_2,\aleph_1}\B$.
  \qed
\end{theorem}

The main result of this paper is that the relations $\simeq^p_{\aleph_1,\aleph_1}$ and
$\equiv_{\aleph_1,\aleph_1}$ may not be equivalent  for structures of size $\aleph_2$.

\begin{theorem}\label{main}
  It is relatively consistent with $\rm ZFC + \rm CH$ that there exist two relational
  structures $\A$ and $\B$ of cardinality $\aleph_2$ in a countable vocabulary
  such that $\A\equiv_{\aleph_1,\aleph_1}\B$ and $\A \not \simeq ^p_{\aleph_1,\aleph_1} \B$.
  \qed
  \end{theorem}

The remainder of the paper is organized as follows. In \S2 we introduce the persistency
game played on a given family of countable partial functions from $\omega_2$ to $\omega_1$.
Given an $(\omega_1,1)$-simplified morass $\morass$ we define a family $\F =\F(\morass)$
which is strategically persistent. If $\morass$ is a generic morass we show that
$\F$ does not have a $\sigma$-closed persistent subfamily.
In \S3 we use the family $\F$ from the previous section to define two structures
$\A$ and $\B$ such that $\A \equiv_{\aleph_1,\aleph_1}\B$. If $\F$ is derived
from a generic morass we show that $\A \not \simeq ^p_{\aleph_1,\aleph_1} \B$.
Finally, in \S 4 we state some open questions and directions for further research.

\section{Persistent families of functions}

In this section we change the original problem and instead of considering the Ehrenfeucht-Fra\"{i}ss\'e
game on a pair of structures, we consider a certain game on a given family of countable
partial functions  from $\omega_2$ to $\omega_1$.

 \def\Fn{\rm Fn}

 Let $\Fn(\omega_2, \omega_1, \omega_1)$ be the collection of all countable partial
functions from $\omega_2$ to $\omega_1$. We say that a subfamily $\F$ of
$\Fn(\omega_2, \omega_1, \omega_1)$ is {\em persistent} if for every $p\in \F$
and $\alpha < \omega_2$ there is $q\in \F$ extending $p$ such that $\alpha \in \dom q$.
We will also consider the following {\em persistency game} on $\F$.

\begin{definition}[$\mathcal G_{\omega_1}(\F)$] \label{persistency}
Suppose $\F$ is a subfamily of $\Fn(\omega_2, \omega_1, \omega_1)$. The game
$\mathcal G_{\omega_1}(\F)$ is played by players $\pla$ and $\ple$ and
runs as follows:

{\setlength\arraycolsep{0.2em}
$$\begin{array}{c|cccccccc}
  \forall & \alpha_0 & &\alpha_1 && \ldots &\alpha_\xi&&\ldots\\
  \hline
  \exists & &p_0 & &p_1 & \ldots &&p_\xi&\ldots\\
\end{array}\hspace{2mm}(\xi<\omega_1)$$
}
\noindent At stage $\xi$ player $\pla$ plays an ordinal $\alpha_\xi< \omega_2$
and  $\ple$ plays $p_\xi \in \F$ extending $p_\eta$, for $\eta <\xi$,
such that $\alpha_\xi \in \dom  {p_{\xi}}$.
$\ple$ wins the game if she is able to play $\omega_1$ moves. Otherwise,
$\pla$ wins.
\end{definition}

We say that $\F$ is {\em strategically persistent} if $\ple$ has
a winning strategy in $\mathcal G_{\omega_1}(\F)$.
One way to guarantee the existence of a winning strategy for $\ple$
is that there exist a persistent subfamily $\mathcal D$ of $\F$ which is
$\sigma$-{\em closed}, i.e. for every sequence $(p_n)_n$ which is increasing
under inclusion there is $q\in \mathcal D$ such that $p_n\subseteq q$, for all $n$.
Indeed, given such a family $\mathcal D$, $\ple$ has a trivial winning
strategy in $\mathcal G_{\omega_1}(\F)$: at stage $\xi$ she plays
any $p_\xi \in \mathcal D$ which extends $\bigcup_{\eta<\xi}p_\eta$ and
such that $\alpha_\xi \in \dom {p_\xi}$.
The main goal of this section is to show that it is relatively consistent
with ${\rm ZFC}$ that there exist a downward closed family $\F$ which is
strategically persistent but does not have a $\sigma$-closed persistent
subfamily. Indeed, given a simplified $(\omega_1,1)$-morass $\morass$
we can read off a certain family $\F=\F(\morass)$ which is strategically
persistent. If $\morass$ is obtained by the standard forcing construction
we show that $\F$ does not have a $\sigma$-closed persistent subfamily.

We start by recalling the relevant definitions from  Velleman~\cite{MR85i:03162}.



\def\ta{\theta_\alpha}
\def\tb{\theta_\beta}

\begin{definition}[\cite{MR85i:03162}]\label{def1} A simplified $(\omega_1,1)$-morass is a pair
$$
\morass=\langle \langle \t_\alpha:\alpha\le\omega_1\rangle,
\langle \F_{\a,\b}:\a<\b\le\o_1\rangle \rangle,
$$
where $\langle \t_\alpha:\alpha\le\omega_1\rangle$ is a sequence of countable ordinals,
$\F_{\a,\b}$ is a family of order preserving embeddings from $\ta$ to $\tb$,  for $\a<\b\le\o_1$,
and the following conditions are satisfied:

\begin{figure}\begin{center}
\begin{picture}(160,95)(20,0)
  \put(25, 95){$\omega_1$}
  \put(20, 45){$\alpha$}
  \put(20, 65){$\beta$}
  \put(30, 25){\line(0,1){60}}
  \put(30, 25){\line(1,0){40}}
  \put(30, 45){\line(1,0){70}}
  \put(30, 65){\line(1,0){100}}
  \put(30, 85){\line(1,0){140}}
  \put(180, 85){$\theta_{\omega_1}=\omega_2$}
  \put(140, 65){$\theta_\beta$}
  \put(110, 45){$\theta_\alpha$}
  \put(30, 0){$\theta_\alpha$ is the $\alpha$-th approximation of  $\omega_2$}
\end{picture}
\end{center}
\caption{A simplified morass.\label{mor}}
\end{figure}

\smallskip
\begin{enumerate}
\item{}{{\bf (Successor)}}
For every $\alpha$ there are $\gamma_\alpha,\eta_\alpha \leq\ta$ such that $\ta = \gamma_\alpha +\eta_\alpha$,
$\theta_{\alpha+1}=\ta + \eta_\alpha$ and
$\F_{\alpha, \alpha + 1} = \lbrace {\rm id}_{\ta}, s_\alpha \rbrace$, where ${\rm id}_{\ta}$ is the identity
on $\ta$ and $s_\a \rest\g_\a={\rm id}_{\gamma_\alpha}$ and $s_\a(\g_\a +\xi)=\ta +\xi$, for
all $\xi<\eta_\alpha$. We call $s_\alpha$ the {\em shift} at $\alpha$. (Figure~\ref{shift}). 

\begin{figure}[h]
\begin{center}
\begin{picture}(220, 50)
  \put(12, 8){$\alpha$}
  \put(2, 38){$\alpha + 1$}
  \put(30, 10){\line(1,0){120}}
  \put(30, 40){\line(1,0){180}}

  \put(90, 10){\line(0,1){30}}
  \put(90, 10){\vector(2,1){60}}
  \put(150, 10){\line(0,1){30}}
  \put(150, 10){\vector(2,1){60}}
  \put(120, 10){\vector(2,1){60}}
  \put(190, 22){$s_\alpha$}

  \put(215, 38){$\theta_{\alpha+1}$}
  \put(160, 7){$\theta_{\alpha}$}

  \put(70, 10){\vector(0,1){30}}
  \put(45, 10){\vector(0,1){30}}
  \put(50, 22){${\rm id}_{\gamma_\alpha}$}
  \put(88, 0){$\gamma_\alpha$}
\end{picture}
\end{center}
\caption{A shift\label{shift}}
\end{figure}

\smallskip

\item{}{{\bf (Composition)}}
If $\alpha < \beta < \gamma$ then
  $\F_{\alpha \gamma} = \lbrace g \circ f : f \in \F_{\alpha \beta},
  g \in \F_{\beta \gamma} \rbrace .$

\smallskip

\item{}{{\bf (Factoring)}} Suppose $\gamma$ is limit, $\alpha <\gamma$ and $f,g \in \F_{\alpha \gamma}$.
Then there exists $\beta$ such that $\alpha <\beta <\gamma$, and $f',g'\in \F_{\alpha,\beta}$
and $h \in \F_{\beta \gamma}$ such that $f = h \circ f'$ and $g = h \circ g'$. (Figure~\ref{factoring}).

\begin{figure}[h]
\begin{center}
\begin{picture}(200, 80)
  \put(10, 10){$\alpha$}
  \put(10, 40){$\beta$}
  \put(20, 5){\line(0,1){65}}
  \put(20, 10){\line(1,0){70}}
  \put(20, 40){\line(1,0){100}}
  \put(20, 70){\line(1,0){140}}
  \put(100, 10){$\theta_\alpha$}
  \put(130, 40){$\theta_\beta$}
  \put(170, 70){$\theta_\gamma$}

  \put(50, 10){\vector(0,1){30}}
  \put(80, 10){\vector(0,1){30}}
  \put(40, 25){$f'$}
  \put(70, 25){$g'$}

  \put(70, 40){\vector(0,1){30}}
  \put(100, 40){\vector(0,1){30}}
  \put(60, 55){$h$}
  \put(90, 55){$h$}
\end{picture}
\end{center}
\caption{Factoring\label{factoring}}
\end{figure}

\item{}{{\bf (Fullness)}}   If $\alpha < \beta$ then $\theta_\beta = \bigcup
\lbrace f[\theta_\alpha] : f \in \F_{\alpha \beta} \rbrace$. Moreover, $\theta_{\omega_1}=\omega_2$.
\end{enumerate}
\end{definition}

We
then have (see \cite{MR85i:03162}) that if $\alpha < \beta\mbox{ and } \xi < \theta_\beta$,
then there is a \emph{unique} predecessor of $\xi$ on level
$\alpha$, i.e. there is a unique $\eta <\theta_\alpha$ such that
$f(\eta)=\xi$, for some $f \in \F_{\alpha \beta}$.
Moreover, any such $f$ is uniquely determined on $\eta + 1$.
We call $\eta$ the $\alpha$-th {\em predecessor} of $\xi$ and write
  $$ \pi_\alpha^\beta(\xi) = \eta. $$

\begin{definition} Given a simplified $(\omega_1,1)$-morass $\morass$
we define  the ordering $\preceq^{\morass}$ on $\omega_2$ as follows:
  $$ \xi \preceq^{\morass}  \eta \quad \mbox{iff}  \quad
   \pi_\alpha^{\omega_1}(\xi) \leq \pi_\alpha^{\omega_1}(\eta), \mbox{ for all } \alpha <\omega_1.$$
We also define the ordering $\preceq_\alpha^{\morass}$ by:
  $$ \xi \preceq_\alpha^{\morass} \eta \quad \mbox{iff} \quad
  \xi \preceq^{\morass} \eta \quad \& \quad
 \pi_\alpha^{\omega_1}(\xi) = \pi_\alpha^{\omega_1}(\eta). $$
If $\morass$ is clear from the context we write $\preceq$ for $\preceq^{\morass}$
and $\preceq_\alpha$ for $\preceq^{\morass}_\alpha$.

\begin{figure}[h]
\begin{center}
\begin{picture}(200, 85)
  \put(20, 5){\line(0,1){65}}
  \put(20, 10){\line(1,0){70}}
  \put(20, 40){\line(1,0){100}}
  \put(20, 70){\line(1,0){140}}
  \put(100, 10){$\theta_\alpha$}
  \put(130, 40){$\theta_\beta$}
  \put(5, 70){$\omega_1$}
  \put(170, 70){$\omega_2$}

  \put(75, 70){\vector(-1,-4){15}}
  \put(95, 70){\vector(-1,-3){20}}
  \put(75, 75){$\xi$}
  \put(95, 75){$\eta$}
  \put(35, 47){$\pi_\beta^{\omega_1}(\xi)$}
  \put(95, 47){$\pi_\beta^{\omega_1}(\eta)$}
\end{picture}
\end{center}
\caption{The ordering $\preceq$.}\label{ordering}
\end{figure}

\end{definition}

Given a simplified $(\omega_1,1)$-morass $\morass$, we define a certain subfamily
 $\F(\morass)$ of $\Fn(\omega_2, \omega_1, \omega_1)$ and show that it is
 strategically persistent.

\begin{definition}\label{F(morass)} Suppose $\morass$ is a simplified $(\omega_1,1)$-morass.
Let $\F(\morass)$ be the set of all  $f \in \Fn(\omega_2, \omega_1, \omega_1)$ such that:
\begin{enumerate}
\item if $\xi,\eta \in \dom f$, $f(\eta)=\alpha$ and $\xi \preceq_\alpha \eta$,
then $f(\xi)=\alpha$.
\item $f^{-1}\{ \alpha \}$ is $\preceq$-bounded, for all $\alpha \in {\ran f}$.
\end{enumerate}
\end{definition}

Note that the family $\F(\morass)$ is closed under subfunctions.
If $\morass$ is clear from the context, we will write $\F$ for $\F (\morass)$.
We first show the following.

\begin{lemma}\label{strategic} Suppose $\morass$ is a simplified $(\omega_1,1)$-morass.
Then $\F(\morass)$ is strategically persistent.
\end{lemma}

\begin{proof} Given $\xi,\eta <\omega_2$, by (4) and (3) of Definition~\ref{def1}
there exists $\alpha <\omega_1$ and $f\in \F_{\alpha,\omega_1}$ such that
$\xi,\eta \in {\ran f}$. Let $\mu(\xi,\eta)$ be the least such $\alpha$.
If $\xi <\eta$ it follows that $\pi_\beta^{\omega_1}(\xi)<\pi_\beta^{\omega_1}(\eta)$,
for every $\beta$ such that $\mu(\xi,\eta) \leq \beta <\omega_1$.
We now describe a strategy for $\ple$ in the persistency game on $\F(\morass)$.
At every stage $j$ if player $\pla$ plays some $\xi_j <\omega_2$ then player $\ple$
 picks an ordinal $\alpha_j <\omega_1$ and plays $p_j=\{ \langle \xi_i,\alpha_i\rangle : i\leq j\}$.
Thus, we only need to describe how to choose the ordinals $\alpha_j$
and check that the corresponding function $p_j$ belongs to $\F(\morass)$.
Suppose we are at stage $j$ and player $\pla$ plays $\xi_j$.
Player $\ple$ first asks if there is an ordinal $i<j$ such that
$\xi_j \preceq_{\alpha_i} \xi_i$. If so, then $\ple$ picks the least such $i$
and sets $\alpha_j=\alpha_i$. Otherwise, $\ple$ picks any ordinal $\alpha_j$
strictly bigger than the $\alpha_i$, for $i<j$, and $\mu(\alpha_i,\alpha_j)$, for $i<j$.
In order to check that the corresponding functions $p_j$ are in $\F(\morass)$
we need the following.

\begin{clai} At every stage $j$ there is at most one $\alpha$ for which
there is $i<j$ such that $\xi_j\preceq_{\alpha}\xi_i$ and  $\alpha_i=\alpha$.
\end{clai}

\begin{proof} Suppose there were two distinct such ordinals, say $\alpha$ and $\beta$.
Let $k$ be the least such that $\alpha_k=\alpha$ and $\xi_j\preceq_\alpha \xi_k$
 and, similarly, let $l$ be the least such that
$\alpha_{l}=\beta$ and $\xi_j \preceq_{\beta}\xi_{l}$.
Suppose that $k< l$. Notice that, by the minimality of $l$,
there is no $i <l$ such that $\alpha_{i}=\beta$
and $\xi_{l}\preceq_{\beta} \xi_{i}$.
Therefore, by the definition of $\alpha_l$, it follows that
$\beta $ is bigger than $\alpha$ and $\mu(\xi_{k},\xi_{l})$.
We consider two cases.

\medskip

\noindent {\em Case 1.} $\xi_k <\xi_{l}$. Since $\beta > \mu(\xi_k,\xi_{l})$
we have that $\pi_{\beta}^{\omega_1}(\xi_k) <\pi_{\beta}^{\omega_1}(\xi_{l})$.
Since $\xi_j \preceq_{\alpha}\xi_k$ and $\alpha <\beta$ we have that
$\pi_{\beta}^{\omega_1}(\xi_j) \leq \pi_{\beta}^{\omega_1}(\xi_{k})$.
Therefore, we have that $\pi_{\beta}^{\omega_1}(\xi_j) < \pi_{\beta}^{\omega_1}(\xi_{l})$.
On the other hand, we have that $\xi_j \preceq_{\beta}\xi_{l}$, which means
that, in particular,  $\pi_{\beta}^{\omega_1}(\xi_j) =\pi_{\beta}^{\omega_1}(\xi_{l})$,
a contradiction.

\medskip

\noindent {\em Case 2.} $\xi_{l} <\xi_k$. Since $\beta >\mu (\xi_k, \xi_{l})$
we have that $\pi_{\gamma}^{\omega_1}(\xi_{l}) <\pi_{\gamma}^{\omega_1}(\xi_k)$, for
all $\gamma \geq \beta$. We also have that
$\pi_{\beta}^{\omega_1}(\xi_j)=\pi_{\beta}^{\omega_1}(\xi_{l})$.
Since  $\xi_j\preceq_{\alpha}\xi_k$ and $\alpha  <\beta$ it follows
that $\xi_{l}\preceq_{\alpha} \xi_k$. Therefore, at stage $l$ we should
have let $\alpha_{l}=\alpha$, a contradiction.
\end{proof}

\noindent Now, we check that the functions $p_j$ belong to $\F(\morass)$, for all $j$.
Condition (1) of Definition~\ref{F(morass)} is satisfied by the construction.
To verify (2), suppose $\alpha \in \ran {p_j}$ and notice that
if $i$ is the least such that $\alpha_i=\alpha$ then $\xi_i$ is the
$\preceq_{\alpha}$-largest element of $p_j^{-1} \{ \alpha\}$.
Therefore, $p_j^{-1} \{ \alpha\}$ is $\preceq$-bounded.
This completes the proof of Lemma~~\ref{strategic}.
\end{proof}

In order to show that $\F(\morass)$ does not have a $\sigma$-closed persistent family
we will need to assume certain properties of $\morass$.

\begin{definition}\label{properties} Let $\morass$ be a simplified $(\omega_1,1)$-morass.
\begin{enumerate}
\item  We say that $\morass$ is {\em stationary} if
${\mathcal S}(\morass) =\{ f[\theta_\alpha] : f\in \F_{\alpha,\omega_1}\}$
is a stationary subset of $[\omega_2]^{\omega}$.
\item  We say that $\morass$ satisfies the  $\aleph_2$-{\em antichain condition} if
for every $X\subseteq (\omega_2)^\omega$ of size $\omega_2$
there are distinct $s,t\in X$ such that $s(n)\preceq t(n)$, for all $n$,
i.e. there is no antichain of size $\aleph_2$ in $(\omega_2,\preceq)^\omega$
under the product ordering.
\end{enumerate}
\end{definition}

We first show that if $\morass$ has the above properties then $\F(\morass)$
does not have a $\sigma$-closed persistency subfamily. Then we show
that if $\morass$ is obtained by the standard forcing for adding
a simplified $(\omega_1,1)$-morass then $\morass$ has the above properties.

\begin{lemma}\label{extension} Suppose {\rm CH} holds and $\morass$ is
a simplified $(\omega_1,1)$-morass which satisfies the $\aleph_2$-antichain condition.
Let $\A$ be a subset of $\F(\morass)^\omega$ of size $\aleph_2$. Then
there is $\vec{g}\in \A$ and  $\B\subseteq \A$ of size $\aleph_2$ such
that for every $\vec{h}\in \B$, every $n$, and every $f\in \F(\morass)$,
if $f$ extends $h_n$ and ${\dom {g_n}} \subseteq {\dom f}$
then $f$ extends  $g_n$.
\end{lemma}

\begin{proof} First, observe that if $X$ is a subset of $(\omega_2)^\omega$
of size $\aleph_2$ then there is $s\in X$ and $Y\subseteq X$ of size
$\aleph_2$ such that $s(n)\preceq t(n)$, for all $t\in Y$ and all $n$.
To see this, let $Z$ be a maximal antichain in $X$. Then every element of $X$ is
comparable with an element of $Z$. Since $\preceq$ refines the usual ordering on $\omega_2$,
by ${\rm CH}$, for every $s\in Z$ the set of $t\in X$
such that $t(n)\preceq s(n)$, for all $n$, has size at most $\aleph_1$.
Therefore, there is $s\in Z$ such that the set
$$
Y=\{ t\in X: s(n)\preceq t(n), \mbox{ for all } n\}
$$
is of size $\aleph_2$. Then $s$ and $Y$ are as required.

We now turn to the proof of the lemma.
First of all, we may assume that there is a fixed ordinal $\alpha <\omega_1$
such that $\alpha = \sup (\bigcup_n {\ran {g_n}})$, for all $\vec{g}\in \A$.
By {\rm CH}, we may assume that there is a fixed ordinal $\mu >\alpha$ and, for each
$n$ a subset $E_n$ of $\theta_\mu$ such that, for every $\vec{g}\in A$,
there is $f_{\vec{g}}\in \F_{\mu,\omega_1}$  such that $ f_{\vec{g}}[E_n]= \dom {g_n}$.
Consider now the functions $e_{n,\vec{g}}= g_n \circ f_{\vec{g}}$, for
$\vec{g}\in \A$ and $n <\omega$. By {\rm CH} again, we may assume that there
are fixed functions $e_n$, such that $e_{n,\vec{g}} = e_n$, for all
$\vec{g} \in \A$ and $n$.
By the first paragraph of this proof, there is $\vec{g} \in \A$ and
a subset $\B$ of $\A$ of size $\aleph_2$ such that $f_{\vec{g}}(\xi)\preceq f_{\vec{h}}(\xi)$,
for all $\vec{h}\in\B$ and $\xi <\theta_\mu$.
We claim that $\vec{g}$ and $\B$ are as required.
To see this, consider some $\vec{h} \in \B$ and some integer  $n$. Let $u$ be any extension  of
$h_n$ which belongs to $\F(\morass)$ and is defined on  $\dom {g_n}$.
We check that $u$  extends $g_n$.
Let $\rho \in \dom {g_n}$. Then there is $\xi \in E_n$ such that $f_{\vec{g}}(\xi)=\rho$.
Let $\rho'=f_{\vec{h}}(\xi)$. Then $\rho \preceq_\mu \rho'$.
Since $u$ extends $h_n$ and $h_n(\rho') \leq \mu$,  by (1) of Definition~\ref{F(morass)} it follows
that $u(\rho)=h_n(\rho')$. On the other hand, $g_n(\rho)=h_n(\rho')=e_n(\xi)$.
Therefore, $u(\rho)=g_n(\rho)$. Since $\rho$ was arbitrary it follows that
$u$ extends $g_n$.
\end{proof}

\def\z{\zeta}

\begin{lemma}\label{sigma-persistent} Assume {\rm CH} and let $\morass$
be a simplified $(\omega_1,1)$-morass which is stationary and satisfies
the $\aleph_2$-antichain condition. Then there is no $\sigma$-closed persistent
subfamily of $\F(\morass)$.
\end{lemma}

\begin{proof} Fix a  persistent subfamily $\G$ of $\F(\morass)$. We need to show
that $\G$ is not $\sigma$-closed.
Let $\tau$ be a sufficiently large regular cardinal.
 Since ${\mathcal S}(\morass)$ is stationary in $[\omega_2]^\omega$,
we can find a countable elementary submodel $M$ of $H_\tau$ containing all the relevant
objects such that $M\cap \omega_2\in {\mathcal S}(\morass)$.
Let $\zeta = \sup (M\cap \omega_2)$ and fix an increasing sequence
$\{\z_n\}_n$ of ordinals in $M$ which is cofinal  in $\z$.

We now work in $M$. For each  $\delta<\omega_2$ fix $g_\delta^0 \in \G$ such that
$\delta \in\dom{g_\delta^0}$.  We can find $\alpha<\omega_1$ and
$X_0\subseteq \omega_2\setminus\zeta_0$ of size $\aleph_2$
such that $g_\delta^0(\delta)= \alpha$, for all $\delta \in X_0$.
Since $\morass$ satisfies the $\aleph_2$-antichain condition, by  Lemma~\ref{extension}
we can fix $\delta_0\in X_0$ and $X_1\subseteq X_0\setminus\zeta_1$
of size $\aleph_2$ such that, for all $\delta \in X_1$, any extension
of $g_\delta^0$ to a function in $\F(\morass)$ which is defined on
$\dom {g_{\delta_0}^0}$ must extend $g_{\delta_0}^0$.
For each $\delta \in X_1$ fix some $g_{\delta}^1\in \G$ which extends
$g_{\delta}^0$ and is defined on $\dom {g_{\delta_0}^0}$.
It follows that $g_{\delta_0}^0\cup g_{\delta}^0\subseteq g_{\delta}^1$.
 By Lemma~\ref{extension} again, we can fix $\delta_1\in X_1$ and $X_2\subseteq X_1\setminus\zeta_2$
 of size $\aleph_2$ such that, for all $\delta\in X_2$ and all $h\in \F(\morass)$,
 if $h$ extends $g_{\delta}^1$ and is defined on $\dom {g_{\delta_1}^1}$ then
 $h$ extends $g_{\delta_1}^1$.  We continue like this and get
 an increasing sequence $(\delta_n)_n$ of ordinals from $M$, a decreasing sequence
 $(X_n)_n$ of subsets of $\omega_2$ of size $\aleph_2$, and,  for each $n$ and
 $\delta \in X_n$, a function $g_{\delta}^n \in \G$ such that:

 \begin{enumerate}
 \item $\delta_n \geq \zeta_n$, for all $n$,
 \item $g_{\delta_n}^n \cup g_{\delta}^n \subseteq g_{\delta}^{n+1}$, for all $\delta \in X_{n+1}$.
 \end{enumerate}

 While the sequence $(\zeta_n)_n$ does not belong to $M$, at each stage
 we need to know only finitely many of the $\zeta_n$. Therefore, we
 can perform each step of the construction inside $M$.
 It follows that $(g_{\delta_n}^n)_n$ is an increasing sequence of functions
 from $\G$ and $g_{\delta_n}^n(\delta_n)=\alpha$, for all $n$.
 The sequence $(\delta_n)_n$ is cofinal in $\zeta$ and, since $M\cap \omega_2\in {\mathcal S}(\morass)$,
 it follows that it is unbounded in the sense of $\preceq$. Therefore, any
 functions which extends $\bigcup_n g_{\delta_n}^n$ violates (2) of Definition~\ref{F(morass)}
 and cannot be in $\F(\morass)$. It follows that $\G$ is not $\sigma$-closed.
\end{proof}

We now consider the standard forcing notion for adding a simplified
$(\omega_1,1)$-morass and show that the generic morass is stationary
and satisfies the $\aleph_2$-antichain condition.
Before we start, it will be convenient to make the following definition.

\begin{definition}\label{full}
For $\beta <\omega_2$ let $I_\beta$ be the interval $[\omega_1 \cdot \beta,\omega_1\cdot {\beta +1})$.
We say that a subset $A$ of $\omega_2$ is $\omega_1$-{\em full} if
$A\cap I_\beta$ is an initial segment of $I_\beta$, for all $\beta< \omega_2$.
\end{definition}

We now state a slight variation of the standard forcing for adding
a simplified $(\omega_1,1)$-morass from \cite{MR85i:03162}.

\begin{definition}[\cite{MR85i:03162}] \label{forc}
The forcing notion $\Pcal$  consists of tuples
$$
p= \langle \langle \t^p_\alpha:\alpha\le\d_p\rangle, \langle \F^p_{\a,\b}:
\a<\b\le\d_p\rangle, A_p,i_p \rangle,
$$
where $\delta_p<\omega_1$, $\langle \theta^p_\alpha:\alpha\leq \delta_p\rangle$ is a sequence
of limit ordinals $<\omega_1$, $\F^p_{\alpha,\beta}$ is a collection of order-preserving mappings
from $\ta^p$ to $\tb^p$, $A_p$ is an $\omega_1$-full subset of $\omega_2$,
$i_p$ is an order preserving bijection between $\theta^p_{\delta_p}$ and $A_p$,
and the following conditions hold:

\begin{enumerate}
\item  $\F^p_{\alpha, \alpha + 1} = \{ {\rm id}_{\ta}, s_\alpha \} $, where $s_\a$ is a shift
as in Definition~\ref{def1} (1). 


\item  If $\alpha < \beta < \gamma\leq \delta_p$ then
  $\F^p_{\alpha, \gamma} = \lbrace g \circ f : f \in \F^p_{\alpha, \beta},
  g \in \F^p_{\beta, \gamma} \rbrace$.

\item  Suppose $\alpha < \gamma\leq \delta_p$, $\gamma$ limit and $f,
g \in \F^p_{\alpha, \gamma}$. Then there is $\beta$ such that $\alpha < \beta < \gamma$
and there are $f', g' \in \F^p_{\alpha, \beta}$ and $h \in \F^p_{\beta, \gamma}$
such that $f = h \circ f'$ and $g = h \circ g'$.

\item  If $\alpha < \beta \leq \delta_p$ then $\theta^p_\beta = \bigcup
\lbrace f[\theta^p_\alpha] : f \in \F^p_{\alpha \beta} \rbrace$.
\end{enumerate}
\noindent  The ordering of $\Pcal$ is defined as follows. We say that $q\leq p$ if
 $\delta_p\leq \delta_q$, $\theta^p_\alpha=\theta^q_\alpha$ for $\alpha \leq \delta_p$,
$\F^p_{\alpha,\beta}=\F^q_{\alpha,\beta}$ if $\alpha < \beta \leq\delta_p$, and
$i_p=i_q\circ h$, for some $h\in \F^q_{\delta_p,\delta_q}$. Note that,
in particular, this means that $A_p \subseteq A_q$.
\end{definition}

\begin{lemma}\label{clo}
Let $(p_n)_n$ be a decreasing sequence of conditions in $\Pcal$. Then
there is  $q\in \Pcal$ such that $A_q= \bigcup_n A_{p_n}$ and
$q\leq p_n$, for all $n$. In particular, $\Pcal$ is $\sigma$-closed.
\end{lemma}

\begin{proof}
Suppose $(p_n)_n$ is a decreasing sequence of conditions in $\Pcal$.
We define the required condition $q$.  We let
$A_q=\bigcup_n A_{p_n}$ and $\delta_q=\sup_n \delta_{p_n}$. Note that,
since the sequence of the $A_{p_n}$ is increasing and each of them
is $\omega_1$-full, then so is $A_q$.  Let $\theta^q_{\delta_q}$
be the order type of $A_q$ and $i_q$ the order preserving bijection
between $\theta^q_{\delta_q}$ and $A_q$.
For $\alpha <\delta_q$ we let $\theta^q_\alpha$ be equal to $\theta^{p_n}_\alpha$,
for any sufficiently large $n$. Also, for $\alpha <\beta <\delta^q$
we let $\F_{\alpha,\beta}^q$ be equal to $\F^{p_n}_{\alpha,\beta}$, for any sufficiently
large $n$. It remains to define the collections $\F_{\alpha,\delta_q}^q$, for
$\alpha <\delta_q$. Fix some $\alpha <\delta_q$ and let $n$ be sufficiently large
such that $\alpha < \delta_{p_n}$. We let
$$
\F_{\alpha,\delta_q}^q= \{   i^{-1}_{q} \circ i_{p_n}\circ f : f\in \F^{p_n}_{\alpha,\delta_{p_n}}\}.
$$
It is straightforward to check that $q =\langle \langle \theta^q_\alpha : \alpha \leq \delta_q\rangle,
\langle \F_{\alpha,\beta}^q : \alpha <\beta \leq \delta_q \rangle, A_q,i_q\rangle$
is a condition and $q \leq p_n$, for all $n$.
\end{proof}

It follows that $\Pcal$ preserves $\omega_1$. We now need a lemma on the compatibility
of conditions in $\Pcal$. First, let us say that two condtions $p$ and $q$ are
{\em isomorphic} if $\delta_p=\delta_q$, $\theta^p_\alpha = \theta^q_\alpha$, for all $\alpha \leq \delta_p$,
and  $\F_{\alpha,\beta}^p=\F_{\alpha,\beta}^q$, for all $\alpha <\beta \leq \theta^p_{\delta_p}$.
If $p$ and $q$ are isomorphic, we say that they are {\em directly compatible}
if there is $r\leq p,q$ such that $\delta_r = \delta_p +1$.
We call such $r$ the {\em amalgamation} of $p$ and $q$.

\begin{lemma}\label{compatibility}
Suppose $p$ and $q$ are two isomorphic conditions in $\Pcal$ such that
$A_p \cap A_q$ is an initial segment of both $A_p$ and $A_q$, and
$\sup (A_p\setminus A_q) <\inf (A_q \setminus A_p)$.
Then $p$ and $q$ are directly compatible.
\end{lemma}

\begin{proof} We define a condition $r$ which is the amalgamation of $p$ and $q$.
For simplicity, set $\delta=\delta_p=\delta_p$ and $\theta_\alpha =\theta^p_\alpha =\theta^q_\alpha$,
for all $\alpha \leq \delta$.
Let $\delta_r=\delta +1$ and $A_r=A_p\cup A_q$. Note that, since $A_p$ and
$A_q$ are $\omega_1$-full, then so is $A_r$. Let $\theta^r_{\delta_r}$
be the order type of $A_r$ and $i_r$ the order preserving
bijection between $\theta^r_{\delta_r}$ and $A_r$.
For $\alpha <\beta \leq \delta$ let $\F^r_{\alpha,\beta}= \F^p_{\alpha,\beta}$.
Let $R=A_p \cap A_q$, let $\gamma$ be the order type of $R$
and $\eta$ the order type of $A_p\setminus A_q$ and $A_q \setminus A_p$.
Since $\sup (A_p \setminus A_q) <\inf (A_q \setminus A_p)$ it follows that
$\theta^r_{\delta_r}= \theta_{\delta} + \eta$.
Let $s: \theta_{\delta}\rightarrow \theta^r_{\delta_r}$  be the shift
of $\theta_{\delta}$ at $\gamma$,
i.e. it is the identity on $\gamma$ and $s(\gamma + \xi) = \theta_{\delta}+\xi$,
for all $\xi <\eta$. We let
$\F^r_{\delta,\delta_r}=\{ {\rm id}_{\theta_{\delta}}, s\}$.
Finally, for $\alpha < \delta$ let
$$
\F^r_{\alpha,\delta_r}= \{ g\circ f : f\in \F^p_{\alpha,\delta}, g\in \F^r_{\delta,\delta_r}\}.
$$
Then $r$ is as required.
\end{proof}

\noindent {\bf Remark} Let $p$ and $q$ be as in Lemma~\ref{compatibility} and let $r$
be the amalgamation of $p$ and $q$. Let $i$  be the order preserving bijection
between $A_p$ and $A_q$. What is important for our purposes is that $r$
forces that $\xi \preceq_{\delta_p} i(\xi)$, for all $\xi \in A_p$.

\begin{lemma}\label{cont}
Let $\alpha < \omega_2$. Then, for every $p\in \Pcal$ there is  $r\leq p$ such that $\alpha\in A_r$.
\end{lemma}

\begin{proof}
Let $\beta$ be such that $\alpha \in I_\beta$. We show that every condition $p$
has an extension $r$ such that $A_r\cap I_\beta$ is a proper extension of $A_p \cap I_\beta$.
Since $A_r\cap I_\beta$ is an initial segment of $I_\beta$, for every $r$, the order
type of $I_\beta$ is $\omega_1$ and $\Pcal$ is $\sigma$-closed,
by iterating this operation countably many times we can find a condition $s \leq p$
such that $\alpha \in A_s$. So, fix some $p\in \Pcal$.
Assume first that $A_p \setminus \omega_1\cdot (\beta+1)$
is non empty and let $\eta$ be its order type.
Note that $\eta$ is a countable ordinal. Let $\mu= \min(I_\beta \setminus A_p)$.
Since $A_p$ is $\omega_1$-full we have that $A_p \cap [\mu, \omega_1 \cdot (\beta +1))=\emptyset$.
Let $\nu = \mu + \eta$
and let $A_q = (A_p \cap \omega_1 \cdot \beta) \cup [\omega_1 \cdot \beta, \nu)$.
Then $A_p$ and $A_q$ have the same order type, $A_p \cap A_q$ is an initial segment
of both of them, and $\sup (A_q \setminus A_p) <\inf (A_p\setminus A_q)$. Also
note that  $A_p \cup A_q$ is $\omega_1$-full.
Let $i_q$ be the isomorphism between $\theta^p_{\delta_p}$ and $A_q$.
Let $\delta_p=\delta_q$, $\theta^p_\alpha = \theta^q_\alpha$, for all $\alpha \leq \delta_p$,
$\F_{\alpha,\beta}^p=\F_{\alpha,\beta}^q$, for all $\alpha <\beta \leq \theta^p_{\delta_p}$.
Then $p$ and $q$ satisfy the assumptions of Lemma~\ref{compatibility}.
Let $r$ be their amalgamation. Then $r\leq p$ and $A_r\cap I_\beta$ is a proper
extension of $A_p \cap I_\beta$, as required.

Assume now that $A_p \subseteq \omega_1 \cdot (\beta +1)$. For simplicity,
let $\delta=\delta_p$ and $\theta_\alpha= \theta^p_\alpha$, for $\alpha \leq \delta$.
Recall that this implies that $\theta_\delta$ is the order type of $A_p$.  Let $\mu = \min (I_\beta \setminus A_p)$.
We are going to define the condition $r$ directly. We let $A_r = A_p \cup [\mu, \mu +\theta_\delta)$.
We let $\delta_r = \delta +1$. We let  $\theta^r_\alpha =\theta_\alpha$, for all $\alpha \leq \delta$
and $\theta^r_{\delta+1}= \theta_\delta \cdot 2$. We let $\F^r_{\alpha,\beta}=\F^p_{\alpha,\beta}$,
for all $\alpha <\beta \leq \delta$. We let $\F^r_{\delta,\delta +1}= \{ {\rm id}_{\theta_\delta}, s\}$,
where ${\rm id}_{\theta_\delta}$ is the identity on $\theta_\delta$ and $s$ is the shift
of $\theta_\delta$ at $0$, i.e. $s(\rho)=\theta_\delta + \rho$, for all $\rho <\theta_\delta$.
For $\alpha <\delta$ we let $\F^r_{\alpha,\delta +1}$ consist of all functions
of the form $g \circ f$, where $f\in \F^p_{\alpha,\delta}$ and $g\in \F^r_{\delta,\delta +1}$.
Finally, let $i_r$ be the order preserving bijection between $\theta_\delta \cdot 2$ and
$A_r$. Then $r$ is an extension of $p$ and $A_r\cap I_\beta$ is a proper extension
of $A_p \cap I_\beta$.
\end{proof}

\begin{lemma}\label{aleph_2-cc} Assume ${\rm CH}$. Then $\Pcal$ satisfies
the $\aleph_2$-chain condition.
\end{lemma}

\begin{proof} Let $\A$ be a subset of $\Pcal$ of size $\aleph_2$. By ${\rm CH}$
we may assume that all the conditions in $\A$ are compatible.
Therefore, we can fix an ordinal $\delta$, a sequence $\langle \theta_\alpha : \alpha \leq \delta\rangle$
and a sequence $\langle \F_{\alpha,\beta}: \alpha <\beta \leq \delta$ such that
every condition $p$ in $A$ is of the form
$p = \langle \langle \theta_\alpha: \alpha \leq \delta\rangle,
\langle \F_{\alpha,\beta}: \alpha \leq \beta \leq \delta \rangle, A_p,i_p \rangle$,
for some $A_p$ of order type $\theta_\delta$ where $i_p$ is the order preserving
bijection between $\theta_\delta$ and $A_p$.
 By ${\rm CH}$ again and the $\Delta$-system lemma, we may find
distinct $p,q\in \A$ such that $A_p \cap A_q$ is an initial segment of both
$A_p$ and $A_q$ and such that $\sup (A_p \setminus A_q) <\inf (A_q \setminus A_p)$.
By Lemma~\ref{compatibility} $p$ and $q$ are compatible, as required.
\end{proof}

Assume ${\rm CH}$. By Lemmas~\ref{clo} and \ref{aleph_2-cc} $\Pcal$ preserves
cardinals. Let $G$ be a $\Pcal$-generic filter over $V$.
For $\alpha <\omega_1$, we let $\theta^G_\alpha$ be equal to $\theta^p_\alpha$,
for any $p\in G$ such that $\alpha \leq \delta_p$. We also let $\theta^G_{\omega_1}=\omega_2$.
For $\alpha <\beta <\omega_1$ we let $\F^G_{\alpha,\beta}$ be equal to
$F^p_{\alpha,\beta}$, for any $p\in G$ such that $\beta \leq \delta_p$.
For $\alpha <\omega_1$ we define:
$$
\F^G_{\alpha,\omega_1} = \{ i_p \circ f : f\in \F^p_{\alpha,\delta_p}, p \in G \mbox{ and }
\alpha \leq \delta_p \}.
$$
It follows that
$$
{\morass}_G = \langle \langle \theta^G_\alpha : \alpha \leq \omega_1 \rangle,
\langle \F^G_{\alpha,\beta}: \alpha <\beta \leq \omega_1 \rangle \rangle
$$
is a simplified $(\omega_1,1)$-morass. Let $\dot{\morass}$ be the canonical
$\Pcal$-name for $\morass_G$.

\begin{lemma}
$\force_{\Pcal} \dot{\morass}\mbox{ is stationary}.$
\end{lemma}

\begin{proof} Suppose $p\force_{\Pcal} \dot{C}$ is a club in $[\omega_2]^{\omega}$.
Set $p_0 = p$. By using Lemmma~\ref{cont} and ~\ref{clo} repeatedly
and the fact that $p$ forces $\dot{C}$ to be unbounded in $[\omega_2]^\omega$,
we can build a decreasing sequence $(p_n)_n$ of conditions in $\Pcal$ and
an increasing sequence $(B_n)_n$ of countable subsets of $\omega_2$ such that
$A_{p_n} \subseteq B_n \subseteq A_{p_{n+1}}$ and
$p_{n+1}\force_{\Pcal} B_n \in \dot{C}$, for all $n$.
Let $q$ be the limit of the sequence $(p_n)_n$ as  in   Lemma~\ref{clo}.
Then $A_q = \bigcup _n A_{p_n} = \bigcup_n B_n$.
Since $\dot{C}$ is forced by $p$ to be closed and $q\leq p$ it
follows that $q\force_{\Pcal} A_q \in \dot{C}$. Since $q\force_{\Pcal} A_q \in {\mathcal S}(\dot{\morass})$
and $\dot{C}$ was arbitrary, it follows that $\dot{\morass}$ is forced to be stationary.
\end{proof}

\begin{lemma}\label{chain-condition} Assume ${\rm CH}$ holds in $V$. Then
$\force_\Pcal \dot{\morass} \mbox{ satisfies the  } \aleph_2$-antichain condition.
\end{lemma}

\begin{proof} Suppose $p\in\Pcal$ forces that $\dot{X}$ is a subset of  $(\omega_2)^\omega$
 of size $\aleph_2$. We can find a subset $S$ of $(\omega_2)^\omega$
 of size $\aleph_2$ and, for each $s\in S$, a condition $p_s\leq p$ such
 that $p_s \force_{\Pcal} s \in \dot{X}$. By Lemma~\ref{cont} we may assume
 that $\ran s \subseteq A_{p_s}$, for all $s$. By ${\rm CH}$ we may
 assume that the conditions $p_s$, for $s\in S$, are all isomorphic.
 Let us  fix an ordinal $\delta$, a sequence $\langle \theta_\alpha : \alpha \leq \delta\rangle$
and a sequence $\langle \F_{\alpha,\beta}: \alpha <\beta \leq \delta \rangle$ such that
every condition $p_s$, for $s \in S$,  is of the form
$p_s = \langle \langle \theta_\alpha: \alpha \leq \delta\rangle,
\langle \F_{\alpha,\beta}: \alpha \leq \beta \leq \delta \rangle, A_{p_s},i_{p_s} \rangle$,
for some $A_{p_s}$ of order type $\theta_\delta$, where $i_{p_s}$ is the order preserving
bijection between $\theta_\delta$ and $A_{p_s}$.
Further, again by ${\rm CH}$, we may assume that there are fixed ordinals $\xi_n <\theta_\delta$, for
$n<\omega$, such that $s(n)= i_{p_s}(\xi_n)$, for all $s\in S$ and all $n$.
By  the $\Delta$-system lemma, we may find
distinct $s,t\in S$ such that $A_{p_s} \cap A_{p_t}$ is an initial segment of both
$A_{p_s}$ and $A_{p_t}$ and such that $\sup (A_{p_s} \setminus A_{p_t}) <\inf (A_{p_t} \setminus A_{p_s})$.
Let $r$ be the amalgamation of $p_s$ and $p_t$.  Then $r\force_{\Pcal}s,t\in \dot{X}$.
By the remark following Lemma~\ref{compatibility} it follows that
that $r \force_{\Pcal} s(n) \preceq_\delta t(n)$, for all $n$.
Therefore $r$ forces that $\dot{X}$ is not an antichain in $(\omega_2,\preceq)^\omega$, as required.
\end{proof}

By putting together the results of this section we obtain the following.

\begin{theorem}\label{persistency-theorem}
It is relatively consistent with ${\rm ZFC + CH}$ that there exist
a downward closed subfamily $\F$ of $\Fn(\omega_2, \omega_1, \omega_1)$ which is
strategically persistent but does not have a $\sigma$-closed persistent subfamily.
\qed
\end{theorem}

\section{The main theorem}

The goal of this section is to prove Theorem~\ref{main}. Before we do
that we show that if $\A \simeq^p_{\aleph_1,\aleph_1} \B$ then
we can find an $\omega_1$-back and forth family $\I$ of partial isomorphisms
between $\A$ and $\B$ with additional closure properties.
Recall that we defined $\I$ to be $\sigma$-closed if  every
increasing sequence $(p_n)_n$ of members of $\I$ has an upper bound in $\I$.
We will say that $\I$ is {\em strongly $\sigma$-closed} if $\bigcup_n p_n \in \I$,
for every such sequence $(p_n)_n$. We will need the following.

\begin{lemma}\label{sigma-closure} Assume ${\rm CH}$ and let $\A$ and $\B$
be two structures of size $\aleph_2$ in the same vocabulary such that
$\A \simeq^p_{\aleph_1,\aleph_1} \B$. Then there is an $\omega_1$-back and
forth set $\mathcal J$  for $\A$ and $\B$ which is strongly $\sigma$-closed.
\end{lemma}

\begin{proof} Let $\I$ be a $\sigma$-closed $\omega_1$-back and forth set of partial
isomorphisms between $\A$ and $\B$. We build another $\omega_1$-back and
forth set $\mathcal J$ which is strongly $\sigma$-closed.
We may assume that the base set of both $\A$ and $\B$ is $\omega_2$.
Since $\I$ consists of countable partial functions from $\omega_2$
to $\omega_2$, by ${\rm CH}$ it follows that it is of cardinality $\omega_2$.
Let us fix an enumeration $\{ p_\alpha : \alpha <\omega_2\}$ of $\I$.
We may assume that the empty function belongs to $\I$ and is enumerated
as $p_0$.
We let $q\in \mathcal J$ if $q$ is a permutation of a countable subset
$D_q$ of $\omega_2$ containing $0$ and the following hold:

\begin{enumerate}
\item if $\alpha \in D_q$ then $\dom {p_\alpha} \cup \ran {p_\alpha} \subseteq D_q$,
\item if $\alpha \in D_q$ and $p_\alpha \subseteq q$ then for every
$\xi \in D_q$ there is $\beta \in D_q$ such that $p_\alpha \subseteq p_\beta\subseteq q$,
and $\xi \in \dom {p_\beta} \cap \ran {p_\beta}$.
\end{enumerate}

Note that if $q\in \mathcal J$ then, by (2) and the fact that $0\in D_q$,
we can find a sequence $(\alpha_n)_n$ of elements of $D_q$ such that
$p_{\alpha_0}\subseteq p_{\alpha_1}\subseteq \ldots \subseteq p_{\alpha_n}\subseteq \ldots$,
and $q=\bigcup_n p_{\alpha_n}$.  Since each $p_{\alpha_n}$ is a countable partial isomorphism from $\A$
to $\B$, then so is $q$. Moreover, since $\I$ is $\sigma$-closed there is
$p\in \I$ such that $q\subseteq p$. Note also that $\mathcal J$ is strongly $\sigma$-closed.
In order to show that $\mathcal J$ has the $\omega_1$-back and forth property
it suffices to show the following.

\begin{clai}
For every $p \in \I$ there is $q\in \mathcal J$ such that $p\subseteq q$.
\end{clai}

\begin{proof} Fix a sufficiently large regular cardinal $\tau$ and a countable
elementary submodel $M$ of $H_\tau$ containing $p$ and the enumeration of $\I$.
Fix an enumeration $\{ \alpha_n : n< \omega\}$ of $M\cap \omega_2$. We define an increasing sequence
$(r_n)_n$ of elements of $\I \cap M$ as follows. Let $r_0=p$.
Suppose we have defined $r_n$. By the fact that $\I$ is an $\omega_1$-back and forth
set and $M$ is elementary, we can find $r_{n+1}\in M\cap \I$ extending $r_n$
such that:

\begin{enumerate}
\item[(a)] $r_{n+1}$ either extends $p_{\alpha_n}$ or is incompatible with it,
\item[(b)] $\alpha_n \in \dom {r_{n+1}}\cap \ran {r_{n+1}}$.
\end{enumerate}

This completes the definition of the sequence $(r_n)_n$. Let $q=\bigcup_n r_n$.
Clearly, $q$ is a permutation of $M\cap \omega_2$, i.e. $D_q=M\cap \omega_2$.
We check that $q\in \mathcal J$. Condition (1) is satisfied by elementary of $M$.
To see that condition (2) is satisfied consider some $\alpha,\xi \in D_q$
such that $p_\alpha \subseteq q$. Let $k$ and $l$ be such that $\alpha =\alpha_k$
and $\xi =\alpha_l$. Choose some $n>k,l$. Then $p_\alpha \subseteq r_n$
and $\xi \in \dom {r_n}\cap \ran {r_n}$. By elementary of $M$ there is $\beta \in D_q$
such that $r_n=p_\beta$. Then $\beta$ witnesses condition (2) for $\alpha$ and $\xi$.
\end{proof}
This completes the proof that $\mathcal J$ is a strongly $\sigma$-closed
$\omega_1$-back and forth set of partial isomorphisms between $\A$ and $\B$.
\end{proof}

We now turn to the proof of Theorem~\ref{main}.
We work in a model of ${\rm ZFC +CH}$ in which there is a simplified $(\omega_1,1)$-morass
$\morass$ which is stationary and satisfies the $\aleph_2$-antichain condition.
Let $\F =\F(\morass)$ be the family defined in Definition~\ref{F(morass)}.
Our plan is to define one structure $\C$ and two distinct elements $a$
and $b$ of $\C$ and let $\A = (\C, a)$ and $ \B = (\C, b)$.
$\C$ will consist of two parts, one is $\omega_2$ with the usual ordering.
Its only role is to ensure certain amount of rigidity of $\C$.
The second part of $\C$ consists of layers indexed by countable subsets
of $\omega_2$. Given  $u\in [\omega_2]^{\omega}$ let
$$
\F_u=\{f\in\F\ :\ \dom f=u\}.
$$
We let $\G_u$ be $[\F_u]^{<\omega}$. Since we wish these structures to be disjoint
and $\emptyset$ belongs to all them, we will replace $\emptyset$ in
$\G_u$ by another object, which we denote by $\emptyset_u$, such that
the $\emptyset_u$ are all distinct.
We still denote the modified structure  by $\G_u$. Let
$\G = \bigcup \{ \G_u : u \in [\omega_2]^{\omega}\}$. For $a\in \G$
we let $u(a)$ be the unique $u$ such that $a\in \G_u$.
The base set of $\C$ will be
$$
C= \omega_2 \cup \G.
$$
We now describe the language of $\C$. First, we will have
two binary relation symbols, $\leq$ and $E$.
The interpretation $\leq^{\C}$ of $\leq$ will be the usual ordering on $\omega_2$.
The interpretation of $E$ is as follows:
$$
(\alpha,a)\in E^{\C} \mbox{ iff } \alpha <\omega_2, a\in \G \mbox{ and } \alpha \in u(a).
$$
This guarantees that any  isomorphism of $\C$ is the identity on $\omega_2$ and
maps each $\G_u$ to itself.
We now put some structure on the $\G_u$. Note that $(\G_u,\Delta)$ is a Boolean
group, where $\Delta$ denotes the symmetric difference. We will keep
only the affine structure of this group, i.e. we want
the automorphisms of $\G_u$ to be precisely the shifts
by some member of $\G_u$, i.e. maps of the form:
$$
x \mapsto x \Delta a,
$$
for some fixed element $a$ of $\G_u$.
In order to achieve this, we will add countably many binary relation symbols
$R_{n,i}$, for $i=0,1$ and $n<\omega$. In each $\G_u$ we will interpret
these relation symbols as follows. First, we index the members of $\F_u$ by elements of
$2^\omega$, say  $\F_u = \{ f_x^u : x\in 2^\omega\}$.
If $a,b\in \G_u$ and $a\Delta b$ is a singleton, say $\{f^u_x\}$,
for each $n$ and $i$, we let
$$
R^{\C}_{n,i}(a,b) \mbox{ if and only if } x(n)=i.
$$
Otherwise, no relation between $a$ and $b$ holds. Also, if $u \neq v$ then
no relation $R_{n,i}^{\C}$ holds between elements of $\G_u$ and $\G_v$.
We also need to connect the different layers of our structure. Suppose $u,v\in [\omega_2]^{\omega}$
and $u\subseteq v$. We define a homomorphism $\pi_{u,v}:\G_v\rightarrow \G_u$
as follows. First, for $f\in \F_u$  we let
 $\pi_{u,v}(\{ f\})=\{ f\restriction u\}$. Then we extend $\pi_{u,v}$ to
a homomorphism of $\G_v$ to $\G_u$.
Note that, in general, $\pi_{u,v}(a)$ may be different from  $\{f\restriction u\ :\ f\in a\}$, since
there may be cancelation, i.e. there could exist $f,f'\in a$ with $f\neq f'$ but $f\restriction u=f'\restriction u$.
Now we add a binary relation symbol $S$ and we let:
$$
S^{\C}(a,b) \mbox{ iff } [a,b\in \G,  u(a) \subseteq u(b) \mbox{ and }  \pi_{u(a),u(b)}(b)=a].
$$
This guarantees the following: if $\tau$ is an automorphism of our structure $\C$
then, for each layer $u$, $\tau$ is the shift by some $a_u\in \G_u$ and
if $u\subseteq v$ then $\pi_{u,v}(a_v)=a_u$.
This completes the definition of the structure $\C$.

Now, we turn to the definition of $\A$ and $\B$. Recall that $\ple$ has
a winning strategy, say $\sigma$, in the persistency game on $\F$.
Consider the play of length $\omega$ in which, at stage $n$, player $\pla$
plays $n$ and player $\ple$ responds by following $\sigma$.
Let $p^*$ be the resulting position after $\omega$ moves and
let $f^*$ be the corresponding function.
So, $f^*\in \F_\omega$.
Now, we introduce a new constant symbol, $c$. Then we let $\A$  be the expansion
of $\C$ obtained by interpreting $c$ as $\emptyset_\omega$ and $\B$
the expansion of $\C$ in which we interpret $c$ as $\{ f^*\}$.

\begin{lemma}\label{9}
$\A \equiv_{\aleph_1,\aleph_1} \B$.
\end{lemma}

\begin{proof} We describe informally a winning strategy for player $\ple$ in
${\rm EF}_{\aleph_1}^{\aleph_1}(\ma,\mb)$. Suppose player $\pla$ starts by
playing $A_0$ and $B_0$, where $A_0$ is a countable subset of $\A$ and
$B_0$ is a countable subset of $\B$.  Since the base sets of $\A$ and $\B$ are the same,
we may assume $A_0=B_0$. Let's call this set $C_0$.
Let $C_0'= C_0\cap \omega_2$ and $C_0^{''}=C_0\cap \G$.
Now, let $U_0=\{ u(a) : a\in C_0''\}$. Then, $U_0$ is a countable collection
of countable subsets of $\omega_2$. Let $u_0= \bigcup U_0$.
Then player $\ple$ simulates a play in the persistency game on $\F$
continuing the play $p^*$ in which player $\pla$ enumerates the elements of $u_0\setminus \omega$
in some order after the $\omega$-th move and $\ple$ uses her winning strategy $\sigma$.
Let $p_0$ be the resulting position and $f_0$ the corresponding function.
Then $f_0\in \F_{u_0}$. Let $\varphi_0$ be the function on $C_0''$
defined by:
$$
\varphi_0(a) = a \Delta \{ f_0\restriction u(a)\},
$$
and let $\psi_0 = \varphi_0 \cup {\rm id}_{C_0'}$.
Note that $\psi_0$ is an involution and $\psi_0(\emptyset_{\omega})=\{ f^*\}$,
since $f_0$ extends $f^*$. Thus, we can consider $\psi_0$ as
a partial isomorphism from $\A$ to $\B$ such that $A_0\subseteq \dom {\psi_0}$
and $B_0\subseteq \ran {\psi_0}$. Player $\ple$ then plays $\psi_0$
as her first move in ${\rm EF}_{\aleph_1}^{\aleph_1}(\ma,\mb)$.

In general, in the $\xi$-th move of ${\rm EF}_{\aleph_1}^{\aleph_1}(\ma,\mb)$
player $\pla$ plays a countable subset $A_\xi$ of $\A$ and a countable
subset $B_\xi$ of $\B$. We may assume that $A_\xi =B_\xi$ and we call
this set $C_\xi$.
We let $C^{'}_\xi=C_\xi \cap \omega_2$ and $C_\xi^{''}=C_\xi \cap \G$.
We let $U_\xi = \{ u(a) : a\in C_\xi^{''}\}$ and
$$
u_\xi = \bigcup \{ u_\eta : {\eta <\xi}\} \cup \bigcup U_\xi.
$$
Player $\ple$ simulates a play $p_\xi$ in the persistency game on $\F$ which
extends the $p_\eta$, for $\eta <\xi$, such that after $\bigcup_{\eta <\xi}p_\eta$
player $\pla$ continues by enumerating in some order the elements of
$u_\xi \setminus \bigcup_{\eta <\xi}u_\eta$ and player $\ple$ plays by following
her strategy $\sigma$.  Let $f_\xi$ be the function
corresponding to $p_\xi$. Notice that $f_\xi$ extends $f_\eta$, for
$\eta <\xi$. Now, let $\varphi_\xi$ be the function defined on $C_\xi^{''}$  by
$$
\varphi_\xi(a)= a\Delta \{ f_\xi \restriction u(a)\}.
$$
Finally, let
$$
\psi_\xi = \bigcup_{\eta <\xi}\psi_\eta \cup \varphi_{\xi} \cup {\rm id}_{C_\xi^{'}}.
$$
It is easy to see that $\psi_\xi$ extends $\psi_\eta$, for $\eta <\xi$.
Since $\sigma$ is a winning strategy for player $\ple$ in the persistency
game on $\F$, player $\ple$ can continue playing like this for $\omega_1$ moves.
Therefore, she has a winning strategy in ${\rm EF}_{\aleph_1}^{\aleph_1}(\ma,\mb)$,
as required.
\end{proof}

\begin{lemma}\label{10}$\ma\not\simeq^{p}_{\aleph_1,\aleph_1}\mb$.
\end{lemma}

\begin{proof} This is similar to the proof of Lemma~\ref{sigma-persistent}.
Suppose $\Omega$ is a $\sigma$-closed family of partial isomorphisms from $\A$
to $\B$ with the back-and-forth property. By Lemma~\ref{sigma-closure}, we
may assume that $\Omega$ is strongly $\sigma$-closed.
Let $\psi$ be a member of $\Omega$. Then, the domain of $\psi$ is a countable subset $A_\psi$ of $\A$
and the range is a countable subset $B_\psi$ of $\B$.
Let $A_\psi^{'}=A_\psi \cap \omega_2$ and let $A_\psi^{''}=A_\psi \cap \G$.
Since $\Omega$ has the back and forth property, it is easy to see that $\psi$
has to be the identity on $A_\psi^{'}$ and preserve the layers of $\G$.
Let $U_\psi =\{ u(a) : a \in A_\psi^{''}\}$.
Since $\Omega$  is also strongly $\sigma$-closed, the set of $\psi \in \Omega$ such that $U_\psi$
is directed under inclusion is dense in $\Omega$.
By replacing $\Omega$ by this set we may assume that $U_\psi$ is directed,
for all $\psi \in \Omega$. Let $u(\psi) =\bigcup U_\psi$, for $\psi \in \Omega$.
For $u\in U_\psi$ let $A_{\psi,u} = A_\psi^{''}\cap \G_u$.
It follows that $\psi \restriction A_{\psi,u}$ has to be the shift
by some element of $\G_u$, say $a_{\psi,u}$. Moreover, if
$u,v\in U_\psi$ and $u\subseteq v$ then $\pi_{u,v}(a_{\psi,v})=a_{\psi,u}$.
Each $a_{\psi,u}$ is finite and since $U_\psi$ is directed under inclusion
and $\psi$ can be extended to a function $\rho$ in $\Omega$ which is
defined on some point of $\G_{u(\psi)}$, it follows that there exists
$a_\psi \in \G_{u(\psi)}$ such that $\psi \restriction A_{\psi,u}$ is the shift by
$\pi_{u,u(\psi)}(a_\psi)$, for every $u\in U_\psi$.
Let $n_\psi$ be the cardinality of $a_\psi$. Note that $n_\psi >0$,
since $\psi(\emptyset_\omega) = \{ f^*\}$, so $\psi$ cannot be
the identity on its domain.  Moreover, since $\Omega$ is $\sigma$-closed
and $n_\psi \leq n_\rho$, for every $\psi,\rho \in \Omega$ such that $\psi \subseteq \rho$,
there is $\psi_0 \in \Omega$ and an integer $n$ such that $n_\psi =n$,
for all $\psi \in \Omega$ such that $\psi_0\subseteq \psi$.
We can replace $\Omega$ by  $\{\psi \in \Omega : \psi_0\subseteq \psi\}$,
so without loss of generality we may assume that $n_\psi =n$, for all $\psi \in \Omega$.

Now, we proceed as in the proof of Lemma~\ref{sigma-persistent}. We fix
a sufficiently large regular cardinal $\tau$.
Since ${\mathcal S}(\morass)$ is stationary in $[\omega_2]^\omega$,
we can find a countable elementary submodel $M$ of $H_\tau$ containing all the relevant
objects such that $M\cap \omega_2\in {\mathcal S}(\morass)$.
Let $\zeta = \sup (M\cap \omega_2)$ and fix an increasing sequence
$\{\z_n\}_n$ of ordinals in $M$ which is cofinal  in $\z$.
We now work in $M$. For each  $\delta<\omega_2$, fix $\psi_{\delta,0} \in \Omega$ such that
$\delta \in u({\psi_ {\delta,0}})$. Let us enumerate $a_{\psi_{\delta,0}}$
as, say $\{ f_{\delta,0}^0,\ldots f_{\delta,0}^{n-1}\}$. We can find $\alpha<\omega_1$ and
$X_0\subseteq \omega_2\setminus\zeta_0$ of size $\aleph_2$
such that $f_{\delta,0}^0(\delta)= \alpha$, for all $\delta \in X_0$.
Since $\morass$ satisfies the $\aleph_2$-antichain condition, by  Lemma~\ref{extension}
we can fix $\delta(0)\in X_0$ and $X_1\subseteq X_0\setminus\zeta_1$
of size $\aleph_2$ such that, for all $\delta \in X_1$, and all $i<n$,
any extension of $f_{\delta,0}^i$ to a function in $\F$ which is defined on
$\dom {f_{\delta(0),0}^i}$ must extend $f_{\delta(0),0}^i$.
For each $\delta \in X_1$ fix some $\psi_{\delta,1}\in \Omega$ which extends
$\psi_{\delta,0}$ and is defined on $A_{\psi_{\delta(0),0}}$.
Then $\psi_{\delta,1}$ must be the identity on $A_{\psi_{\delta(0),0}}^{'}$
and
$$
\pi_{u({\psi_{\delta,0}}),u({\psi_{\delta,1}})}(a_{\psi_{\delta,1}}) = a_{\psi_{\delta,0}}.
$$
Since $a_{\psi_{\delta,1}}$ has the same size as $a_{\psi_{\delta,0}}$, we can enumerate it as
$\{ f_{\delta,1}^0,\ldots f_{\delta,1}^{n-1}\}$ such that $f_{\delta,1}^i$
extends $f_{\delta,0}^i$, for all $i<n$. Moreover, $f_{\delta,1}^i$ is defined
on $\dom {f_{\delta(0),0}^i}$ and so it must extend $f_{\delta(0),0}^i$.
In other words, $f_{\delta(0),0}^i\cup f_{\delta,0}^i\subseteq f_{\delta,1}^i$,
for all $i<n$. It follows that $\psi_{\delta,1}$ extends $\psi_{\delta(0),0}$,
for all $\delta \in X_1$.
By Lemma~\ref{extension} again, we can fix $\delta(1)\in X_1$ and $X_2\subseteq X_1\setminus\zeta_2$
of size $\aleph_2$ such that, for all $\delta\in X_2$ and all  $i<n$,
any extension of $f_{\delta,1}^i$ to a function in $\F$ which is defined on
$\dom {f_{\delta(1),1}^i}$ must extend $f_{\delta(1),1}^i$.
For each $\delta \in X_2$ fix some $\psi_{\delta,2}\in \Omega$ which extends
$\psi_{\delta,1}$ and is defined on $A_{\psi_{\delta(1),1}}$.
As before, $\psi_{\delta,2}$ must be the identity on $A_{\psi_{\delta,2}}^{'}$
so it must agree with $\psi_{\delta(1),1}$ on $A_{\psi_{\delta(1),1}}^{'}$
Also, we can enumerate $a_{\psi_{\delta,2}}$ as $\{f_{\delta,2}^0,\ldots,f_{\delta,2}^{n-1}\}$
such that $f_{\delta(1),1}^i \cup f_{\delta,1}^i \subseteq f_{\delta,2}^i$.
We conclude that $\psi_{\delta,2}$ extends $\psi_{\delta(1),1} \cup \psi_{\delta,1}$, for all $\delta \in X_2$.
Continuing in this way we get  an increasing sequence $(\delta(k))_k$ of ordinals from $M$,
a decreasing sequence   $(X_k)_k$ of subsets of $\omega_2$  of size $\aleph_2$, and,  for each $k$ and
$\delta \in X_k$, $\psi_{\delta,k} \in \Omega$ and an enumeration $\{f_{\delta,k}^0,\ldots,f_{\delta,k}^{n-1}\}$
of  $a_{\psi_{\delta,k}}$ such that:

 \begin{enumerate}
 \item $\delta(k) \geq \zeta_k$, for all $k$,
 \item $\psi_{\delta(k),k} \cup \psi_{\delta,k} \subseteq \psi_{\delta,k+1}$, for all $\delta \in X_{k+1}$.
\item $f_{\delta (k),k}^i \cup f_{\delta,k}^i \subseteq f_{\delta,k+1}^i$, for all  $i<n$ and all $\delta \in X_{k+1}$.
\end{enumerate}

Now, $(\psi_{\delta(k),k})_k$ is an increasing sequence of members of $\Omega$ and since $\Omega$
is $\sigma$-closed there is $\rho \in \Omega$ extending all the $\psi_{\delta(k),k}$.
It follows that there is an enumeration $\{ f^0,\ldots, f^{n-1}\}$ of $a_\rho$
such that $f_{\delta(k),k}^i \subseteq f^i$, for each $i<n$ and $k<\omega$.
Recall that $f_{\delta,0}^0(\delta)=\alpha$, for all $\delta \in X_0$. Moreover,
$f_{\delta,0}^i \subseteq \ldots \subseteq f_{\delta,k}^i$, for all $i<n$ and $\delta \in X_k$.
It follows that $f^0(\delta(k))=\alpha$, for all $k$.
However, all the $\delta(k)$ belong to  $M\cap \omega_2$ and the sequence $(\delta(k))_k$
is cofinal in $\zeta$. Since $M\cap \omega_2$ belongs ${\mathcal S}(\morass)$ it follows
that this sequence is $\preceq$-unbounded. Therefore, $f^0$ violates condition (2)
of Definition~\ref{F(morass)} and so it cannot belong to $\F$, a contradiction.
\end{proof}

This completes the proof of Theorem~\ref{main}.
\qed

\section{Open questions}

We mention a couple of questions which remain open.

\begin{question}
 Is it consistent that $\equiv_{\aleph_1,\aleph_1}$ and $\simeq^p_{\aleph_1,\aleph_1}$
are equivalent for structures of size $\aleph_2$ in the context of {\rm CH} ?
\end{question}

\begin{question}
 Is it consistent that $\pisoaa$ is not transitive? This
would show that $\pisoaa$ is not the right concept, i.e. it does not
represent equivalence in some logic.
\end{question}

\bibliographystyle{plain}

\bibliographystyle{plain}
   \def\Dbar{\leavevmode\lower.6ex\hbox to 0pt{\hskip-.23ex \accent"16\hss}D}
  \def\cprime{$'$} \def\ocirc#1{\ifmmode\setbox0=\hbox{$#1$}\dimen0=\ht0
  \advance\dimen0 by1pt\rlap{\hbox to\wd0{\hss\raise\dimen0
  \hbox{\hskip.2em$\scriptscriptstyle\circ$}\hss}}#1\else {\accent"17 #1}\fi}

\end{document}